\documentclass[english,11pt]{article}
\usepackage{authblk}
\usepackage[T1]{fontenc}
\usepackage[latin9]{inputenc}
\usepackage{amscd}
\usepackage{amsthm,amsmath,amsfonts,amssymb}
\usepackage{verbatim}
\usepackage{lmodern}
\usepackage{graphicx}
\usepackage{enumerate}
\usepackage{color}

\makeatletter
\numberwithin{equation}{section}
\numberwithin{figure}{section}
\theoremstyle{plain}
\newtheorem{thm}{\protect\theoremname}
  \theoremstyle{plain}
  \newtheorem{lemma}[thm]{\protect\lemmaname}
    \newtheorem*{lemma*}{Lemma}

    \newtheorem{prop}[thm]{\protect\propname}

\newtheorem*{thm*}{Theorem}

\numberwithin{thm}{section}

\theoremstyle{remark}
\newtheorem*{rem}{Remark}
\makeatother

\usepackage{babel}
\providecommand{\propname}{Proposition}

\providecommand{\lemmaname}{Lemma}
\providecommand{\theoremname}{Theorem}

\newcommand{\imag}{\operatorname{Im} \,}
\newcommand{\real}{\operatorname{Re} \,}
\renewcommand{\Im}{\imag}
\renewcommand{\Re}{\real}

\newcommand{\ee}{\epsilon}

\newcommand{\E}{\mathbf{E}}

\newcommand{\Prob}{\mathbf{P}}

\newcommand{\hcap}{\operatorname{hcap}}
\newcommand{\diam}{\operatorname{diam}}

\newcommand{\dist}{\operatorname{dist}}
\newcommand{\Z}{{\mathbb Z}}

\newcommand{\C}{{\mathbb C}}
\newcommand{\rad}{{\rm rad}}

\newcommand{\Q}{{\bf Q}}

\newcommand{\SLE}{\text{\tiny SLE}}
\newcommand {\G} {{\mathcal G}}

\newcommand{\LERW}{\text{\tiny LERW}}

\def \Half {{\mathbb H}}
\def \Disk {{\mathbb D}}
\def \F {{\cal F}}
\newcommand {{\wind}} {{\rm wind}}
\newcommand {{\Sine}}{S}
\let \setminus \smallsetminus
\let \le \leqslant
\let \leq \leqslant
\let \ge \geqslant
\let \geq \geqslant
\let \epsilon \varepsilon
\let \phi \varphi

\newcommand{\whoknows}  {{\mathcal A}}

\newcommand{\saws}{\mathcal{W}}
\newcommand {\eset}{{\emptyset}}

\AtEndDocument{\bigskip{\footnotesize%
  (Lawler) \textsc{Department of Mathematics, University of Chicago} \par  
  \textit{E-mail address}: \texttt{lawler@math.uchicago.edu} \par
  
  \addvspace{\medskipamount}
  (Viklund) \textsc{Department of Matematics, KTH Royal Institute of Technology} \par
  \textit{E-mail address}: \texttt{fredrik.viklund@math.kth.se}
}}

\title{Convergence of radial loop-erased random walk in the natural parametrization}
\author{Gregory F. Lawler}
\affil{University of Chicago}

\author{Fredrik Viklund}
\affil{KTH Royal Institute of Technology}

\begin{document}
\maketitle
\begin{abstract}
In recent work we have shown that loop-erased random walk (LERW) connecting two
boundary points of a domain converges to the chordal
Schramm-Loewner evolution (SLE$_2$) 
in the sense of curves parametrized by Minkowski content.  In this note we explain
how to derive the analogous result for LERW from a boundary point
to an interior point, converging towards radial SLE$_2$. 
\end{abstract}
\section{Introduction and main results}
\subsection{Introduction}
Let $D$ be a bounded, simply connected domain containing $0$ as an interior point. Let $a,b \in \partial D$ and suppose $\partial D$ is analytic near $a,b$. For large integer $N$, let $D_{N}$ be an approximation of $D$ using the grid $N^{-1} \Z^{2}$ with $a_{N}, b_{N} \in \partial D_{N}$ approximating $a,b$. In \cite{LV_LERW_natural, LV_lerw_chordal_note} we proved that a loop-erased random walk (LERW) in $D_{N}$ from $a_{N}$ to $b_{N}$ viewed as a continuous curve and parametrized so that each edge is traversed in time a constant times $N^{-5/4}$ converges in the scaling limit to a chordal SLE$_{2}$ curve in $D$ from $a$ to $b$ parametrized by $5/4$-dimensional Minkowski content. In short, LERW converges to SLE in the natural parametrization. The version of LERW studied in \cite{LV_LERW_natural, LV_lerw_chordal_note} is defined by taking a random walk from $a_{N}$ to $b_{N}$ conditioned on staying in $D_{N}$ and successively erasing loops as they form to produce a random \emph{chordal} self-avoiding walk. Another natural variant, \emph{radial} LERW, is the loop-erasure of a random walk from an interior point stopped when reaching the boundary. (A third version that we do not directly discuss, is the loop-erasure of a random walk in the whole plane using a limiting procedure, and corresponds to the whole-plane version of SLE$_2$.) In this note we will continue the work of \cite{LV_LERW_natural, LV_lerw_chordal_note} by proving an analogous result
 about radial LERW converging to radial SLE$_2$. The idea of the proof is to use the Markovian coupling of \cite{LV_LERW_natural, LV_lerw_chordal_note} and weight it by the relevant Radon-Nikodym derivatives in order to construct a coupling of the radial processes up to the first time the paths get near the target point. We conclude by giving separate ``continuity'' estimates for LERW and SLE near the target point in the interior. Another possible strategy would be to redo the work of \cite{LV_LERW_natural, LV_lerw_chordal_note} in the present setting, but this would need the analogue of the sharp one-point estimate of \cite{BLV} which we currently do not have for the radial version of LERW. We will start by giving a statement
 of the main result and a sketch of the argument and then
 discuss the details.
  
 \subsection{Notation}
 In order to state our results we give some notation and describe the set up. We will use
 some notations of \cite{LV_LERW_natural, LV_lerw_chordal_note} occasionally reminding the
 reader of their meanings.
 
 Let $\mathcal{A}$ be the set of bounded, simply connected subsets of $\Z^2$ containing $0$ as an interior vertex. To each $A \in \mathcal{A}$ we associate to a simply connected Jordan domain $D_A \subset \C$ obtained
 by replacing each vertex with the closed square with axis-paralell sides of length one
 centered at the vertex, and then taking the interior. We will refer to the domains $D_A$
 as ``union of squares'' domains; they are in one-to-one correspondence
 with elements of $\mathcal{A}$. 
 Suppose boundary edges $a,b \in \partial_e A$ are given; we write $a,b$ both for the edge and for its
 midpoint which is a point in $\partial D_A$.
   Let $\saws_{A,a,0}$ be the
 set of self-avoiding walks (SAW) $\eta = [\eta_0,\ldots,\eta_k]$
 with $\eta_k= 0$, $\{\eta_1,\ldots,\eta_{k-1}\} \subset A$
 and such that $a $ is the midpoint of the edge $[\eta_0,\eta_1]$. (We use the notation $\eta$ only for SAWs.)
   We also view $\eta$ as the  continuous
 curve $\eta(t): 0 \leq t \leq k - \frac 12$ with $\eta(0) = a$
  obtained by traversing
 the edges of $\eta$ at unit speed. Throughout the paper, if $\gamma:[0,t_\gamma] \rightarrow \C$ is a continuous curve, we also write
 $\gamma$ for the curve $\gamma:[0,\infty) \rightarrow \C$ extended in the natural manner as $\gamma(t)
  = \gamma(t_\gamma), \, t > t_\gamma$.  
  We define a finite measure on $\saws_{A,a,0}$ by
 \begin{equation}\label{def:radial-lerw-measure}       \hat  P_{A, a,0}(\eta) =  p(\eta)\, \Lambda_\eta(A) , \quad \eta \in \saws_{A,a,0} \end{equation}
 where $\log \Lambda_\eta(A)$ is the random walk loop measure of loops staying in $A$
 that intersect $\eta$, see \cite[Chapter 9]{LL}, and $p(\cdot) = 4^{-\# \text{steps}(\cdot)}$ is the random walk measure on the square grid.  Then
 \[   \sum_{\eta \in  \saws_{A,0,a}}   \hat   P_{A,a,0}(\eta)  = H_A(0,a), \]
 where $H_A$ denotes the Poisson kernel for random walk, that is, the probability that simple random walk from $0$ exits $A$ through the edge $a$. 
 We write \[\Prob_{A,a,0}^\text{rad}(\eta) =  \frac{\hat  P_{A,a,0}(\eta) }{H_A(0,a)}, \quad \eta \in \saws_{A,a,0}, \] for the probability
 measure obtained by normalization. This is the probability distribution of radial LERW in $A$ from $a$ to $0$. $\Prob_{A,a,0}^\text{rad}$ induces a probability measure on parametrized complex curves from $a$ to $0$ in $D_A$ after parametrizing by arclength. The distributions on other versions of LERW are given by considering different sets of SAWs and normalizing appropriately.

 Suppose now that $D$ is a simply connected domain in $\C$
containing the origin and assume $D$ has analytic
boundary.  For each positive integer $N$ we
define $D_N = D_{A_N}$ to be the largest simply
connected ``union of
squares'' domain
containing the origin whose closure is contained in $N \cdot D$
and we write $\check D = \check D_N := N^{-1} D_N$ which satisfies
$0 \in \check D \subset D$.  As $N \to \infty$, the simply connected domain $\check{D}$ converges to $D$ in the Carath\'eodory sense. Given $a_N \in \partial_e
 A_N$ we write
 $\check a= \check a_{N}:= N^{-1} a_N$. In order to state the theorem, assume that $a_{N} \in \partial D_{N}$ is a sequence chosen so that $\check{a} \to a$ as $N \to \infty$ and for each $N$, let $\Prob_N^\text{rad}$ be the probability
measure  obtained from  $\Prob_{A_N,a_N,0}^\text{rad}$ by
considering the scaled paths
 \begin{equation}\label{lerw-param} \check \eta(t) = \check \eta_N(t) := N^{-1} \,\eta(c_* t \,N^{5/4}),
 \;\;\; 0 \leq t \leq t^N_* := \frac{k + \frac 12}{c_* \, N^{5/4}}. \end{equation}
 (We use the same notation both for the measure on SAWs and parametrized curves.) Here $c_* \in (0, \infty)$ is
a fixed constant whose value is not known; it is the same constant as the $c_*$ appearing in \cite{LV_LERW_natural}. In general, given a SAW $\eta$, we will write $\check \eta_N$ for the rescaled path parametrized as in \eqref{lerw-param}. Let $\gamma(t), \, 0 \le t \le t_{\gamma}$ be radial SLE$_{2}$ in $D$ from $a$ to $0$ parametrized by $5/4$-dimensional Minkowski content; we write $\mu^{\text{rad}}=\mu_{D,a,0}^{\text{rad}}$ for its law.
\subsection{Statements}
  With these notations in place the main theorem can be stated as follows.
\begin{thm}\label{thm:main.intro2} As $N \to \infty$, the law of $\check{\eta}$ converges weakly to that of $\gamma$ with respect to the metric $\rho$.

More precisely, for each $\ee > 0$ there exists $N_{0} < \infty$ such that if $N > N_{0}$ then is a coupling of $\gamma$ with distribution $\mu^{\text{rad}}$ and $\check \eta$ with distribution $\Prob^{\text{rad}}_N$ such that
\[
\Prob \{\rho(\gamma, \check \eta) > \ee \} < \ee.
\]
\end{thm}
Here we are using a metric $\rho$ on parametrized curves defined as follows. Given curves  $\gamma^{j}:[s_{j}, t_{j}] \to \mathbb{C},\, j=1,2,$  we let
\begin{equation}\label{metric}
\rho(\gamma^{1}, \gamma^{2}) = \inf \Big\{ \sup_{s_1 \le t \le t_1 }|\alpha(t)-t| + \sup_{s_1 \le t \le t_1}|\gamma^{1}(t) - \gamma^{2}(\alpha(t))| \Big\},
\end{equation}
where the infimum is taken over increasing homeomorphisms (``reparametrizations'') $\alpha : [s_{1}, t_{1}] \to [s_{2}, t_{2}]$.
\begin{rem}
The metric $\rho$ is convenient to work with but it is not the only natural choice. It has the disadvantage that the metric space of parametrized curves $\gamma: [0,t_\gamma] \to \C$ with the metric $\rho$ is not complete. An alternative would be to consider curves $\gamma: [0, t_\gamma] \to \C$ as elements $\gamma = (t_\gamma, t \mapsto \gamma(t \wedge t_\gamma)) \in [0, \infty) \times \mathcal{C}[0,\infty)$ with a metric \[\hat \rho \left\{ \gamma^1, \gamma^2\right\}= |t_1-t_2| + \sup_{t \in [0,\infty)}|\gamma^1(t \wedge  t_1) - \gamma^2(t \wedge  t_2)|, \]
if  $\gamma^j = (t_j, t \mapsto  \gamma^j(t \wedge t_j)), \, j=1,2.$   
\end{rem}
By a straightforward
estimate of the conformal map $f: \check D \rightarrow
D$ with $f(\check a ) = a, f(0)=0$, one can show that the two curves
$\gamma$ and $f\circ \gamma$ are close in the
sense of \eqref{metric}.  
See \cite[Corollary 7.3]{LV_LERW_natural} for a proof
for the analogous chordal result.  Therefore, 
it suffices to prove the corresponding result where $\gamma$
is replaced with $\check \gamma(t), \, 0 \le t \le t_{\check \gamma}$, 
 a radial SLE$_2$ curve from $\check a$ to $0$
in $\check D$.   The main work is in establishing
the following proposition that considers paths stopped before
they get too close to the origin.
\begin{prop}\label{prop:main.intro}  There exists $c < \infty$
such that for every $r > 0$ and all $N$
sufficiently large, we can define $\check \gamma(t),
 0 \leq t \leq t_{\check \gamma},$ and $\check\eta(t), 0 \leq t \leq t_N^*$,
on the same probability space $( \Omega, \Q)$
 so that the following holds.
      \begin{enumerate}
\item The marginal distribution of $\check \gamma$ is that of radial SLE$_2$ from
$\check a$ to $0$ in $\check D$ parametrized by $5/4$-dimensional Minkowski content.
\item  If $\Q_2$ denotes the
marginal distribution on $\check \eta$, and $\Prob_N^{\text{rad}}$
the distribution of (scaled) LERW from $\check a$ to $0$ in $\check{D}$ parametrized as in \eqref{lerw-param},
  then
$$\| \Q_2 - \Prob_N^{\text{rad}}\| \le c \, r,$$ where
$\|\cdot\|$ denotes total variation distance.
 \item  There is a Markovian stopping time $\tau$ for the coupled pair $(\check \gamma, \check \eta)$  and an event  $E$  with  $ \Q(E) 
\geq 1- cr$ on which
\[ 
   \max_{0 \leq t\leq \tau}
  |\check \gamma(t) - \check \eta(t)| \leq r,\]
 and
 \[ |\check \gamma(\tau)| \le r. \]
\end{enumerate} 
    \end{prop}
By Markovian stopping time we mean, roughly speaking, that the domain Markov property is valid at the stopping time for each path individually. See Section~\ref{sect:coupling} for a precise statament. 
Given this proposition, the proof of Theorem~\ref{thm:main.intro2} follows from Proposition~\ref{prop:continuity} below. Indeed, using Proposition~\ref{prop:main.intro} we couple the paths until time $\tau$ when they are at most at distance $2r$ from $0$. Given the paths up to this time, we can then extend them independently to paths from $\check a$ to $0$. Proposition~\ref{prop:continuity} shows that the ends of the paths do not increase the $\rho$-distance much.
 
 The estimates we give here are not optimal
but they more than suffice for our purposes. The notation in the statement is the same as in Proposition~\ref{prop:main.intro}.
\begin{prop}\label{prop:continuity} There exists $c < \infty$
such that on the probability space of Proposition~\ref{prop:main.intro}, except perhaps on an event of probability
$c\, r^{1/8} $,
\[     t_{\check \gamma} - \sigma_r \leq c\, r^{1/2};\]
\[t_*^N - \sigma_r' \leq c \, r^{1/2};\]
 \[    \sup_{\sigma_r \le t \le t_{\check \gamma}}\big|\check \gamma(t)\big| \leq r^{1/2},\]
and
 \[   \sup_{\sigma_r' \le t \le t_*^N}\big|\check \eta(t)\big| \leq r^{1/2},\]
 where
 \[
 \sigma_{r} = \inf\{t \ge 0 : |\check \gamma(t)|\le r\}, \quad \sigma_{r}' = \inf\{t \ge 0 : |\check \eta(t)|\le r\}.
 \]
   \end{prop}
\begin{proof}See   
 Proposition~\ref{nov30.lemma1} and Proposition~\ref{radial-reg}.
\end{proof}
\subsection*{Acknowledgements}
Lawler was supported by National Science Foundation grant DMS-1513036. Viklund was supported by the Knut and Alice Wallenberg Foundation, the Swedish Research Council, the Gustafsson Foundation, and National Science Foundation grant DMS-1308476. We also wish to thank the Isaac Newton Institute for Mathematical Sciences and Institut Mittag-Leffler where part of this work was carried out.
\section{Comparing radial and chordal}\label{sect:radialvschordal}
We will need to compare measures on paths with different target points -- radial and chordal measures, both for LERW and SLE.
\subsection{Radial and chordal SLE}\label{sect:rvscsle}

We review what is known comparing chordal and radial SLE$_2$. (Most
of this generalizes to other values of $\kappa$ but
for ease we will restrict to $\kappa = 2$.) 
We will consider the paths in $\check{D}$ with $N \ge 10$. (This restriction on $N$ is arbitrary, but we want to avoid some trivialities for small $N$.) All constants in this section will be uniform in $N \ge 10$ but may depend on $D,a,b$.  Given  a simple curve $\gamma$ from $a$ to $b$ in $D$, we write \[S_t = S_{D \setminus  \gamma_t}(0;  \gamma(t),  b)\] where for a choice of conformal map $F_t:   D \setminus   \gamma_t \to \Half$, $F_t(  \gamma(t)) = 0$, $F_t( {b}) = \infty$, we define
 \[
 S_t(z) = S_{D \setminus \gamma_t}(z;   \gamma(t),  b) = \sin(\arg F_t(z)).
 \]
Let $\mu^{\operatorname{rad}},\mu^{\operatorname{chord}}$ denote
the probability measures of radial SLE$_2$ from $\check a$ to $0$ and
chordal SLE$_2$ from $\check a$ to $\check b$ in $\check D$, respectively.  We view
these as probability measures on curves $\check \gamma$, parametrized
by Minkowski content,  stopped
at stopping times $\tau$ such that $\check \gamma(\tau) \in \check D \setminus
\{0\}$, that is, stopped before the paths reach $0$ or $\check b$. Note that the content is a deterministic function of the path, so we could equally well consider the capacity parametrization.
Write $h_D(0,b)$ for the Poisson kernel and $h_{\partial D}
(a,b)$ for the boundary Poisson kernel. 
We normalize so  that 
\begin{equation}  \label{nov8.1}       h_{\partial D}(a,b)
     =  S_{D,a,b}^{-2} \, h_D(0,a) \, h_D(0,b) .
     \end{equation}
The radial and chordal laws are absolutely continuous, in fact, if we write $\check{D}_\tau = \check{D} \setminus \check\gamma_\tau$, 
 we have
  \begin{equation}  \label{nov8.3}
   \frac{d \mu^{\operatorname{rad}} }{d\mu^{\operatorname{chord}} } \, (\check\gamma_\tau)= M^{\SLE}_\tau :
       =  \frac{S_{\check D_\tau}(0;  \check \gamma(t), \check b)^2 }
         {S_{\check D}(0,\check{a}, \check b)^2} \cdot \frac{h_{\check{D}}(0,\check b)}{h_{\check{D}_\tau}(0, \check b)} .
         \end{equation}
    It is convenient to write the last factor on the right-hand side in coordinates:
   \begin{equation}\label{eq:imags}
   \frac{h_{\check{D}}(0,\check b)}{h_{\check{D}_\tau}(0,\check b)} = \frac{\Im F(0)}{\Im F_{\tau}(0)}.
   \end{equation}
   Here $F: \check{D} \to \Half$, $F(a) = 0, \, F(b) = \infty$ and $F_{\tau} :  \check{D}_{\tau} \to \Half, \,F_{\tau}(\check{\gamma}(\tau))=0, \, F_{\tau}(\check b)=\infty$, where these maps have the same normalization, i.e., $F_{\tau}\circ F^{-1}(z) = z + o(1)$ as $z  \to \infty$. 
 
These statements can be made precise by first mapping to the upper
   half-plane and using the Girsanov theorem; see, e.g.,
   \cite{SW_coordinate}.
   We state this as a proposition.
  
  \begin{prop} 
   For every $r>0$ suppose
  that $\tau$ is a 
   stopping time such that
  $\dist(0, \check \gamma_\tau) \geq r$ and $\dist(
  \check b,\check  \gamma_\tau) \geq r$.  Suppose
  that  $\check  \gamma(t),
  0 \leq t \leq  \tau,$ has the distribution
  of  chordal
  SLE$_2$ from $ \check a$ to $  \check  b$ stopped
  at time $\tau$ with respect to the measure $\mu^{\operatorname{chord}}$.
  If the measure $\nu$ is defined by the relation \[d\nu = M_\tau^{\SLE} \, d\mu^{\operatorname{chord}},\] then under
  $\nu$, $\check \gamma(t),
  0 \leq t \leq  \tau,$ has the distribution
  of  radial
  SLE$_2$ from $ \check a$ to $0$ stopped
  at time $\tau$.
  
  \end{prop}

\subsection{Radial and chordal loop-erased walks}\label{sect:rvsclerw}
The laws of the LERW aiming towards different points stopped appropriately can be compared in a similar manner as the SLE$_{2}$.
 Let $A \in \mathcal{A}$ with $a,b \in \partial_{e}A$ given. If $\eta = [\eta_0,\ldots,\eta_{k}]$ is a SAW starting at $a
 =[\eta_0,\eta_1]$ and otherwise staying in $A$, 
 we write $A_\eta=A \setminus \eta$, taking the component of $0$ if necessary.  Suppose 
$a' = [\eta_k,w]$ is an edge such that $w
 \in A_\eta$. Then $a'$ is a boundary edge of $A_\eta$.
We  write $\Prob_{A,a,0}(\eta \oplus a')$ for the probability that the beginning
 of the LERW from $a$ to $0$ is given by $\eta
  \oplus a'$.  Then,  
 \[    \Prob_{A,a,0}(\eta \oplus a') = 
4^{-|\eta|} \, \Lambda_\eta(A) \, \frac{H_{A _ \eta}(0, a' )}
  {H_A(0,a)}, \]
  where $H_{A _ \eta}(0, a' )=0$ if $w$ and $0$ are not in the same component of $A_\eta$. 
Let $b$ be a boundary edge of $A$. Then if $\Prob_{A,a,b}$ denotes the probability
  measure for LERW from $a$ to $b$ in $A$, we have a similar expression,
  \[    \Prob_{A,a,b}(\eta \oplus a') = 
4^{-|\eta|} \, \Lambda_\eta(A) \, \frac{  H_{\partial A _ \eta}(a',b )}
  {H_{\partial A} (a,b)}.\] 
We can write this as a ``Radon-Nikodym derivative'', in which the loop terms have cancelled:
\[      \frac{\Prob_{A,a,0}(\eta\oplus a')}{ \Prob_{A,a,b}(\eta\oplus a')}
  = \frac
  {H_{A _ \eta}(0, a' )}  { H_A(0,a)}
   \; \frac {H_{\partial A} (a,b) }  
   {  H_{\partial A _ \eta}(a',b ) }  \]   
   provided that the denominator is nonzero. 
Given $\eta, a'$ as above, we define  
   \[
   M^\LERW(\eta \oplus a') =  M^\LERW_{A,0,a,b}(\eta \oplus a')=\frac
  {H_{A _ \eta}(0, a' )}  { H_A(0,a)}
   \; \frac {H_{\partial A} (a,b) }  
   {  H_{\partial A _ \eta}(a',b ) }
   \]
   which naturally defines a martingale with respect to LERW in $A$ from $a$ to $b$. 
\begin{prop}  Let $\Prob_{A,a,b}$ be the law of
 LERW from $a$ to $b$ in $A$.
Suppose $\sigma$ is a stopping time   such that
 $ H_{\partial A _ {\eta'}}(b, a' )
 >0$,
where $\eta =[\eta_0,\ldots,\eta_{\sigma -1}]
, a' = [\eta_{\sigma-1},\eta_\sigma].$  If the measure
${\bf Q}_{A,a,0}$ is defined by the relation \[d{\bf Q}_{A,a,0} = M^\LERW(\eta \oplus a') \, d\Prob_{A,a,b},\]
then under ${\bf Q}_{A,a,0}$, $\eta \oplus a'$ has the
distribution of LERW  from $a$ to $0$ in $A$
stopped at $\sigma$.
\end{prop}

An important fact proved in \cite{KL} (Eq. (41)) is that there is a constant $\tilde{c}$ such that
 \begin{equation}  \label{nov8.2}
   H_{\partial A}(a,b) = \tilde c \, H_A(0,a) \, H_A(0,b)
  \, S_{A,a,b}^{-2} \, [1 + O_{\delta}(N^{-u})]
  \end{equation}
  if $S_{A,a,b} \geq \delta$. (This also writes the Green's function next to the boundary as a constant times the boundary Poisson kernel.) Here we are writing $S_{A,a,b}$ for $S_{D_A, a,b}(0)$ as defined in the previous subsection.
    If $\dist(\eta, 0) \ge r N$, then we can write
  \begin{equation}  \label{nov8.4}
    M^{\LERW}(\eta \oplus a')=\frac{\Prob_{A,a,0}(\eta\oplus a')}{ \Prob_{A,a,b}(\eta\oplus a')}
= \frac{S_{ A_\eta,a',b}^2}
        {S_{ A,a,b}^2} \, \frac{H_A(0,b)} 
           {H_{A_\eta}(0,b)} 
            \, [1 + O_{r,\delta}(N^{-u})], 
            \end{equation}
   provided that the sines are bounded below by $\delta > 0$. This is very similar to \eqref{nov8.3}.

Let $I_{s}$ be the set of SAWs $\eta$ in $A$ such that $\eta \cap C_{s} \neq \emptyset$. Let $I_s^*$ be the set of SAWs 
 $\eta = [\eta_0,\ldots,\eta_n] \in I_s$ such
 that $[\eta_0,\ldots,\eta_{n-1}] \not \in I_s$.
 
   \begin{lemma}\label{lerwrnlemma}Let $r, \delta > 0$ be given. For all $A, \eta$ with $\eta \in I_{s}^{*}$ such that
   \begin{equation}\label{condition111}
s \ge r N, \quad  \min_{j \leq n}
   S_{A_{j},a_{j},b} \ge \delta, \quad A_{j}=A \setminus \eta[0, \ldots \eta_{j}],
\end{equation}
we have
\[\frac{H_{A}(0,b)}{H_{{A_{\eta}}}(0,b)} = \frac{\Im F(0)}{\Im F_{\eta}(0)}\left(1+O_{r, \delta}(N^{-1/20}) \right).\]
Here  $F_{\eta}: D_{A_{\eta}} \to \mathbb{H}$ is the conformal map with $F_{\eta}(a_{\eta}) = 0, F_{\eta}(b) = \infty$ and such that $F_{\eta} \circ F^{-1}(z) = z + o(1)$  as $z \to \infty$. The error estimate is uniform over $A,\eta$ satisfying \eqref{condition111}.
\end{lemma}
\begin{proof}
See Section~\ref{sect.lemma1proof}.
\end{proof}

\section{Regularity estimates}
This section gives estimates on the content of the SLE and LERW near the target point and estimates for the geometric ``sine processes'' $S_{t}$ and $S_{A_{j},a_{j},b}$. 

\subsection{SLE estimates}

 We know that the expected Minkowski content of a chordal SLE$_2$ curve is finite almost surely, at least if the domain is bounded and not too rough. We need a similar
 estimate for $\mu^{\operatorname{rad}}$ near the bulk point.

\begin{lemma}\label{radialcontent} There exists $c < \infty$ such that
if $\check \gamma$ is 
a radial SLE$_2$ path from $\check a$ to $0$,  $T_{r,s}$
denotes the $(5/4)$-Minkowski content of $\gamma
\cap \{s < |z| < r \}$, and $T_r = T_{0+,r}$, then
\begin{equation}  \label{dec2.1}
  \E[T_r] \leq c \, r^{5/4} . 
  \end{equation}
\end{lemma}
\begin{proof}
We have
\begin{align*}
\E[T_{r}] & = \int_{r\mathbb{D}}G_{\check D}^{\operatorname{rad}}(z; \check{a}, 0) dA(z) \le c \int_{r\mathbb{D}}G_{\mathbb{D}}^{\operatorname{rad}}(z; 1,0) dA(z), 
\end{align*}
where $G_{D}^{\operatorname{rad}}(z,a,0)$ denotes the Green's function for radial SLE$_{2}$ in $D$ from $a$ to $0$. From \cite{AKL} we have that
\begin{equation}\label{radial-green}
G_{\mathbb{D}}^{\operatorname{rad}}(z; 1,0)  = c \, e^{2y(2-d)}[\sinh y \cosh y]^{d-2 + \beta}|\sin (x+iy)|^{-\beta} \E^{*}_{z}[g'_{T}(0)^{q}].  
\end{equation}
Here in the special case $\kappa=2$, the boundary exponent $\beta=3$, $d=5/4$, and $q=-3/4$, the cylindrical coordinates $(x,y)$ are defined by $z = e^{-2y + 2ix}$, and the expectation is with respect to two-sided radial SLE in $\mathbb{D}$ from $1$ through $z$ with Loewner maps $(g_{t})$ and with $T$ being the time (in the radial parametrization) at which the path reaches $z$. Since the conformal radius satisfies $g'_{T}(0)^{-1} \le 4 |z|$ under $\mathbb{P}^{*}_{z}$,  we can see that \eqref{radial-green} implies
\[
G_{\mathbb{D}}^{\operatorname{rad}}(z; 1,0) \le c |z|^{-3/4}, 
\]
as $|z| \to 0$, and integrating this around $0$ gives the stated bound.
\end{proof}
 \begin{prop}  \label{nov30.lemma1}
 There exists $c < \infty$ such
 that the following holds. 
 Let $s \leq 1/4$ and $\sigma_s = \inf\{t: |\check \gamma(t)|
  = s\}$.  Then, except for an event of probability at
  most $c \, s^{1/8}$, we have \[t_\gamma- \sigma_s \leq   s^{1/8}\]
  and \[|\check \gamma(t)| \leq s^{1/2}\] for $t \geq \sigma_s$.
 \end{prop}

 \begin{proof} 
 By the estimates of \cite{lawler_continuity,
  lawler_field} the probability
   that $\check \gamma$ leaves $s^{1/2}\mathbb{D} $
  after time $\sigma_s$  is
  $O(s^{3/4})$.  On the event that
  \[\gamma[\sigma_s,t_\gamma] \subset s^{1/2}\mathbb{D},\]
  we have $t_\gamma - \sigma_s \leq T_{s^{1/2}}$
  where $T_r$ is as in \eqref{dec2.1}.
Since by Lemma~\ref{radialcontent}, $\E[T_{s^{1/2}}]
 \leq c\, s^{5/8}$, Markov's inequality implies $\Prob\{T_s \geq s^{1/2} 
  \}  \leq c\, s^{1/8}.$ 
 \end{proof}

Recall that if
$f_t: \check D \setminus \check \gamma_t \rightarrow \Disk$ is the
conformal transformation with $f_t(0) = 0, f_t(\check \gamma(t)) = 1$,
then $\theta_t$ is defined by $f_t(\check b) = e^{2i\theta_{t}}$ and $S_{t} = S(\check \gamma_{t})= \sin \theta_{t}$; this definition works for any Loewner curve. We will also need the fact that for SLE, $S_t$ is unlikely to get
near zero in a finite amount of time.

 \begin{lemma}  \label{lemma1.nov9}
 Let $D,a,b$ be given. For every $r > 0$, there
 exists $\delta_r > 0$ such that
 if $\gamma$
 is a radial  SLE$_2$ path from $\check a$ to $0$
 in $\check D$, then
\begin{equation}  \label{nov9.2}
  \Prob\left\{\min_{0 \leq t \leq \sigma_r}
          S_t \leq \delta_r\right\} \leq r,
          \end{equation}
          where $\sigma_r = \inf\{t: |\check \gamma(t)|  \leq r\}$.
 \end{lemma}

\begin{proof}[Sketch of proof.] This is standard so we will give a sketch. Consider the process parametrized by the radial
parametrization centered at $0$: at time $s$ of the radial parametrization, the
conformal radius is
$         e^{-2s} \, r_{\check{D}}.$
 Let $f_{t}$ be as just above.

The Loewner equation shows that $ \theta_t$ satisfies
  the radial Bessel SDE
  \[     d  \theta_t = \cot   \theta_t \, dt + d W_t , \]
  where $W_t$ is a standard Brownian motion.  The 
  Schwarz lemma
  implies that the conformal radius at
  time $\sigma_r$ at least $r$, and hence $\sigma_r$
  corresponds to time at most $-\frac 12 \, \log r + O(1)$.
  Hence, this becomes an estimate about the radial Bessel
  SDE which is standard.  Indeed, we could 
   give a more quantitative version of \eqref{nov9.2}, but
   we will not need it. 
    \end{proof}
      We will later work with stopping times that correspond to mesoscopic capacity increments, we need to be able to stop the paths slightly
 later than the stopping time $\sigma_{2r}$ and we will want
 to say that this time is smaller than $\sigma_{r}$ given some geometric information.  For this, we will
 use the following deterministic lemma. 
 
 \begin{lemma}\label{lemma2.nov9} Let $D,a,b$ be as above. For every $0 < r < 1/4$ and $\delta >0$
  there exists $c >0$ depending on $D,a,b,r,\delta$ such
 that the following is true. Suppose $ \gamma$
 is a simple radial  Loewner curve from $a$ to
 $0$ in $ D$.   If
 \[  \min_{0 \leq t \leq \sigma_r}
          S_{D_t}(0; \gamma(t),b) \geq  \delta,\]
 then the following estimates hold:
 \begin{enumerate}
 \item[i)]{
 \[         h_{ D_{\sigma_r}}(0,   b)
    \geq  c; \]}
    \item[ii)]{
 \[  \hcap[\gamma_{\sigma_r} \setminus \gamma_{\sigma_{2r}}]
   \geq c;\]}
   \item[iii)]{
   \[
   \dist_I(\gamma_{\sigma_{r}}, b) \ge c,
   \]
   }
   \end{enumerate}
where $\dist_I$ denotes interior distance.
\end{lemma}

\begin{proof} We use the radial
parametrization from $0$ as above.    As before, let
$f_t = g_t \circ f_0$, and if we write $\sigma_{r}$ for the first time $\gamma$ reaches the circle of radius $r$ about $0$, we have
\[           \frac{h_{D_{\sigma_r}}(0, b)}
   {h_D(0,  b)} = |g_{\sigma_r}'(e^{2i\theta_0})|.\]
An easy estimate using the Loewner equation
shows that there exists $C_\delta <\infty$
such that if $S_s \geq \delta$ for $0 \leq s \leq \sigma_r$, then
\begin{equation}\label{g-estimate}      |g_{\sigma_r}'(e^{2i\theta_0})|  \geq c e^{-rC_\delta}.\end{equation}

This gives the first estimate.  For the second estimate, recall that half-plane capacity (viewed from $b$) can be given as a hitting measure for Brownian bubbles  attached at $b$. More precisely, up to a constant depending on the map to $\Half$ we use to measure capacity, the capacity of $\gamma[\sigma_{2r}, \sigma_{r}]$ equals  
\[
 |f'_{0}(b)|^{2}|g_{\sigma_{2r}}'(\zeta)|^{2} \lim_{\ee \to 0+} \frac{1}{\ee} \, \Prob^{\zeta,\ee}\left (B \cap f_{\sigma_{2r}}(\gamma[\sigma_{2r}, \sigma_{r}]) \neq \emptyset \right),
\]
where $\zeta =e^{2i\theta_0}$ and $\Prob^{\zeta, \ee}$ denotes the law of a Brownian $h$-process $B$ in $\mathbb{D}$ from $(1-\ee)\zeta$ to $\zeta$ and we also used the covariance rule. But since $f_{\sigma_{r}}(\gamma[\sigma_{2r}, \sigma_{r}])$ is a curve in $\mathbb{D}$ connecting the boundary with a circle of radius at most a constant strictly smaller than $1$, the limit in the last display is bounded away from $0$. Therefore \eqref{g-estimate} gives the result. 

For the last estimate, suppose that $\ee > 0$ and let $\tau$ be the first time $\gamma$ is at interior distance $\ee$ from $b$. Then there is a crosscut $\beta$ of $D \setminus \gamma_{\tau}$ connecting $\gamma(\tau)$ with $b$ and $\diam \beta \le \ee$. Moreover, $\beta$ cuts $D \setminus \gamma_{\tau}$ into two components, one of which contains $0$. By the Beurling estimate, the harmonic measure of $\beta$ from $0$ in $D \setminus \gamma_{\tau}$ is at most a constant times $\ee^{1/2}$ and the same bound (allowing a different constant) holds for $S_{D_\tau}(0;\gamma(\tau),b)$. 

\end{proof}
  
\subsection{LERW estimates}  \label{lemmasec}

We start with an estimate that does not require any smoothness
of the boundary of $A$ and then we use it for the  approximations $\check D$
of the analytic domain $ D$. For $r > 0$, let $C_r = \{z \in \Z^2: |z| < r\}$. 

\begin{lemma}   \label{nov30.lemma2} There exists $c < \infty$ such that the
following holds.  Suppose $A$ is a simply connected finite
subset of $\Z^2$ containing the origin and $a \in \partial_e A$.
If $z \in A$, let $d_z = d_{z,A} = \dist(z,\partial A)$ and let
$d = d_0$. 
\begin{itemize}
\item  If $z \in C_{d/4}$, then
\[     \Prob_{A,0,a}^{\operatorname{rad}}\{z \in \eta\}
   \leq  c\,  |z|^{-3/4}. \]
\item  If $z \in A \setminus C_{d/4}$, then
\begin{equation}  \label{sep13.4}
    \Prob_{A,0,a}^{\operatorname{rad}}\{z \in \eta\}
   \leq  c\, \frac{H_A(z,a)}{H_A(0,a)}
     \, G_A(z,0)\, d_z^{-3/4}. 
     \end{equation}
   \end{itemize}
\end{lemma}

\begin{proof}
Fix $z \in A$ and let 
\[  r = r_z = \frac 14 \cdot \min\{d_z, |z| \}, \]
\[    D_{r}(z) =   C_r + z = \{w: |w-z| < r \}, \;\;\;\;
  D_{2r}(z) = C_{2r} + z =  \{w: |w-z| < 2r \}.\]

Let $\omega^a,\omega^0$ be two independent 
simple random walks starting
at $z$ ending  at $a,0$, respectively, and otherwise
staying in $A$.  The ($p$-) measure of the set of such $\omega^a$ is
$H_A(z,a)$ and the measure of such $\omega^0$ is 
$G_A(z,0)$. 
\begin{itemize}
\item Let $\tilde \omega^0$ be $\omega^0$
with the first edge removed (so that it starts at
a nearest neighbor of $z$). 
\item Let $\eta^a = LE(\omega^a)$.
\end{itemize}
Then $\hat P_{A,0,a}\{z \in \eta\}$ equals the $p$-measure
of $(\omega^a,\omega^0)$  satisfying
\[     \tilde \omega^0 \cap \eta^a = \eset. \]
Let $\hat \eta^a$ be $\eta^a$ stopped at the first time
that it reaches $\partial D_{r}(z)$ and similarly for $\hat \omega^{0}$ using $\tilde \omega^{0}$.
 Then   $\hat P_{A,0,a}\{z \in \eta\}$
is bounded above by 
the measure
of $(\omega^a,\omega^{0} )$  satisfying
\[
   \hat  \omega^{0}  \cap \hat \eta^a = \eset. 
\] 

By \cite[Lemma 6.13]{LV_LERW_natural}, the normalized
probability measure on $\hat \eta^a$ is comparable to
the measure obtained by taking LERW in the larger set $D_{2r}(z)$ from
$z$ to $\partial D_{2r}(z)$ (and stopping when exiting $D_{r}(z)$). If we let $\pi_{2r}$ denote the latter
probability, then the measure of $\omega^a$ that
produce $\hat \eta^a$ is comparable to $H_A(z,a)
 \, \pi_{2r}(\hat \eta^a)$.
 Using 
  \cite[Proposition 6.15]{LV_LERW_natural}, we see that
 that
 \[   \E_{\pi_{2r}} \left[\Prob\{ \hat  \omega^0 \cap \hat \eta^a = \eset
  \mid \eta^a\}\right] \le c\, r^{-3/4}. \]

 For a given $\tilde \omega^0$, the measure of possible extensions
 to walks ending at the origin and avoiding $\hat \eta^a$
  is $G_{A  \setminus \hat \eta^a}(\zeta,0)$ where
 $\zeta$ is the endpoint of  $\tilde \omega^0$ which is a neighbor of $z$.  We now consider
 two cases.
 \begin{itemize}
 \item  If $ z \in A \setminus C_{d/4}$, we bound
  $G_{A  \setminus \hat \eta^a}(\zeta,0)$ above by
  $G_A(\zeta,0)$ and use the Harnack inequality to see
  that this is comparable to $G_A(z,0)$.
 \item  If $z \in C_{d/4}$ we bound
 $G_{A  \setminus \hat \eta^a}(\zeta,0)$ above by
 $G_{\Z^2  \setminus \hat \eta^a}(\zeta,0)$. 
 Since $\hat \eta^a$ is a set of diameter at least
 $r$ and $|\zeta| \geq |z|/2$,
 we can use standard arguments to show that
 \[  G_{\Z^2  \setminus \hat \eta^a}(\zeta,0) < c.\]
 In this case, we also see that the Harnack inequality implies
 that $H_A(z,a) \asymp H_A(0,a)$.
 \end{itemize}
 
\end{proof}

\begin{prop}  There exists $c = c_D < \infty$ such that
for all $(A_n,a_n,b) \in \whoknows_n(D)$,
\[    \sum_{z \in A} \Prob_{A_n,0,a_n} \{z \in \eta\}
  \leq c \, n^{5/4}. \]
\end{prop}

\begin{proof}  
We will prove a stronger result.  We first see that if $
|z| \leq n/2$, then $\Prob_{A_n,0,a_n} \{z \in \eta\}
 \leq c \, |z|^{-3/4}$, and hence for $m \leq n/2$,
\begin{equation}  \label{sep13.2}
   \sum_{m \leq |z| \leq m+1} 
 \Prob_{A_n,0,a_n} \{z \in \eta\} \leq c \, m^{1/4}.
 \end{equation}
 
 Let $U_k = \{z \in A:  d_z  \geq  k\}$.  Using the
 smoothness of $D$ and the gambler's ruin estimate, 
 we see that
\begin{equation}  \label{sep13.5}
         G_A(z,0) \leq  \frac{c \, k}{n}, \;\;\;
    z \in (\partial U_k) \setminus C_{n/2} \end{equation}
 Let $S$ be a simple random walk starting at the origin
 and let $T_k = \inf\{j: S_j \in U_k\}.$  Using smoothness
 of $\partial D$ and gambler's
 ruin again, we see that for $w \in \partial U_k$,
 $H_{U_k}(0,w) \asymp n^{-1}$.  Also, the strong Markov
 property implies that
 \[      H_A(0,a) =  \sum_{w \in \partial U_k }
        H_{U_k}(0,w) \, H_{A}(w,a).\] 
        Hence,
  \[ \sum_{w \in \partial U_k }
        H_{A}(w,a)  \leq c\,n\,  H_A(0,a).\]
        If $k-1 \leq d_z < k$, then there exists $w$
     within distance two of $z$ with $d_{w} \geq k$
     and hence either $z$ or a nearest neighbor of $z$
     is in $\partial U_k$. Using the Harnack inequality,
   we can then conclude that
   \[      \sum_{k-1 \leq d_z < k }
        H_{A}(z,a)  \leq c\,n\,  H_A(0,a).\] 
 Hence, \eqref{sep13.4} and \eqref{sep13.5}
 imply that 
 \[  \sum_{k-1 \leq d_z < k } 
   \Prob_{A_n,0,a_n} \{z \in \eta\}
     \leq  
   c \, k^{ 1/4} . \]
 This and \eqref{sep13.2} give the estimate.
 \end{proof}
 
 \begin{prop} \label{radial-reg} There exists $c < \infty$ such
that the following holds.  Let $A$ be a finite, simply connected
domain including the origin, $a \in \partial_e A$,
 and  $n = \dist(0,
\partial A).$  Suppose $\eta= [\eta_0,\ldots,\eta_k] 
$ is a LERW from $a$ to $0$
in $A$,    $s \leq 1/2$, 
and we write  
\[  \eta = \eta^- \oplus \eta^+\]
 where $\eta^- =[\eta_0,
\ldots,\eta_m]$ and $m$ is the smallest index with $\eta_m
 \in C_{s^2n}$.    Then, except
for an event of probability at most $c s$,
\[  |\eta^+| \leq  s^{1/4} \, n^{5/4}, \;\;\;\;
\eta^+ \subset C_{  s n}.\]
\end{prop}

\begin{proof} We write $\Prob,\E$ for probabilities
and expectations under $\Prob_{A,a,0}.$
Let $T$ be the total number of points in $C_{ s n}$ visited
by $\eta$.  By Lemma \ref{nov30.lemma2}, we have for $z
\in C_{sn}$, $\Prob\{z \in \eta\} \leq c \, |z|^{-3/4}$,
and hence 
$  \E[T] \leq c \, ( s n)^{5/4}  $
and  
\[   \Prob\{T \geq s^{1/4} \, n^{5/4} \} \leq c \, s.\]
On the event $\{\eta^+ \subset C_{ s n}\}$, we have
$|\eta^+| \leq T$.  Therefore it suffices to show that
$   \Prob\{\eta^+ \not\subset C_{sn}\} \leq c \, s.$
We will show the stronger fact,
\[   \Prob\{\eta^+ \not\subset C_{sn} \mid \eta^-\} \leq c \, s.\]

Given $\eta^-$, the distribution of $\eta^+$
 can be given as follows:
 \begin{itemize}
 \item  Take a simple random walk $\omega$ starting at $\eta_{m}$ and stop
 it when it reaches the origin. Let $\tilde \omega$ be $\omega$ with
 the first edge removed stopped at the first visit to the origin. 
 \item  Condition on the event that  $\tilde \omega  \subset A
 \setminus \eta^-$.
 \item Erase loops from the path.
 \end{itemize}
 Hence it suffices to show that 
 \[   \Prob\{ \tilde \omega 
 \not\subset 
   C_{s n} ,   
 \tilde \omega  \subset A
 \setminus \eta^- \mid \eta^-\}
  \leq c \, s \, \Prob\{ \tilde \omega  \subset A
 \setminus \eta^- \mid \eta^-\}.\]
 
 Let $q$ be the maximum over $w \in C_{3s^2n/4} \setminus
 C_{s^2n/4}$ of the probability that a random walk starting
 at $w$ reaches the origin before leaving $A \setminus
 \eta^m$.  Standard estimates show that this comparable
 to $[\log (s^2 n)]^{-1}$ and the minimum 
 over $w \in C_{3s^2n/4} \setminus
 C_{s^2n/4}$ is also comparable
 to this. 
 Let $\omega'$ denote $\tilde \omega$ stopped at the first
 time the path gets distance $s^2 n /2$ from $\eta_{m}$,
 and let $z$ denote the endpoint of $\omega'$.  By 
 \cite[Lemma 6.9]{LV_LERW_natural},
 \[  \Prob\{|z| \leq 2s^2 n/3 \mid   \omega'
\cap \eta^s \neq \eset\} \geq c , \]
and hence
\[     \Prob\{ \tilde \omega  \subset A
 \setminus \eta^- \mid \eta^-\}  \geq c \,q \, \Prob\{
 \tilde \omega
\cap \eta^- \neq \eset\}.\]
The Beurling estimate implies that the probability that
a walk starting in $C_{2s^2n}$ reaches $\partial C_{sn}$ and then
returns to $C_{2 s^2 n}$ without hitting $\eta^-$ is $O(s)$ (it
is $O(\sqrt s)$ to reach $\partial C_{sn}$  and given that,
 another
$O(\sqrt s)$ to return to $C_{2 s^2 n}$).
 Given that it does this, the probability of reaching the origin
before leaving $A \setminus \eta^s$ is bounded above by $q$.
 Hence,
\begin{eqnarray*}
   \Prob\{ \tilde \omega 
 \not\subset 
   C_{s n} ,   
 \tilde \omega  \subset A
 \setminus \eta^- \mid \eta^-\} & \leq  & 
 c \, s \, q \, \Prob\{
 \tilde \omega
\cap \eta^- \neq \eset\}\\
& \leq & c \,s\,  \Prob\{ \tilde \omega  \subset A
 \setminus \eta^- \mid \eta^-\} .
 \end{eqnarray*}

\end{proof}
 We end with the analogue of Lemma~\ref{lemma1.nov9}. 
 \begin{lemma}\label{aug7.lemma3}
  Let $D,a,b$ be given. For every $r > 0$, there
 exists $\delta_r > 0$ such that
 if $\check \eta$
 is scaled LERW in $\check{D}$ from $\check a $ to $0$, then for all $N$ sufficiently large,
\begin{equation} 
  \Prob\left\{\min_{0 \leq t \leq \sigma_r}
          S_{\check D \setminus \check \eta_{t}, \check \eta(t), \check b} \leq \delta_r\right\} \leq r,
          \end{equation}
          where $\sigma_r = \inf\{t: |\check \eta(t)|  \leq r\}$.
 \end{lemma}
 \begin{proof}
 This follows from Lemma~\ref{lemma1.nov9} and the fact that we know that radial LERW converges to radial SLE$_{2}$ in the radial parametrization \cite{LSW04}.  \end{proof}

\section{Coupling: proof of Proposition~\ref{prop:main.intro}}\label{sect:coupling} In this section we construct the coupling of Proposition~\ref{prop:main.intro}. 
Recall the setup:  we have a fixed domain $D$ with analytic boundary with $0 \in D$ and with $a\in \partial D$ given. We assume without loss in generality that $1 \le r_D(0) \le 2$. Fix an auxiliary boundary point $b \in \partial D$ such that if $F_0: D \to \Half$ is the conformal transformation with $F_0(a) = 0, \, F_0(b) = \infty$, $\Im F_0(0) = 1$, then  $|\Re F_0(0) | \le 1$. The first three conditions can always be fulfilled, but the last puts a restriction on $b$ to not be ``too close'' to $a$.

For $N=1,2, \ldots$ let $\check{D}_N$ be a discrete approximation of $D$ as above, and we always choose $\check{a}_N \in \partial \check{D}_N$ to correspond to a boundary edge of $A_N$ closest to $N \cdot a $ and $\check{b}_N$ is chosen to approximate $b$ with the same geometric condition fulfilled: $F_N: \check{D}_N \to \Half, \, F_N(\check a_N)=0, \, F_N(\check b_N)=\infty$ and $\Im F_N(0) = 1, |\Re F_N(0)| \le 1$.  We will write $\check a = \check{a}_N, \check D = \check{D}_N$ etc, and keep the $N$ implicit in the notation from now on. When we measure capacity of curves in $\check D$ it is with respect to the map $F=F_N$.

We will construct our coupling by weighting the coupling of \cite{LV_LERW_natural, LV_lerw_chordal_note} by the Radon-Nikodym derivatives discussed in Section~\ref{sect:radialvschordal}. Because of this we need to be careful about measurability properties and use the Markovian property of the chordal coupling and we need to consider events on which the Radon-Nikodym derivatives are controlled. Let us review the chordal coupling. Fix  $K < \infty$. 
In \cite{LV_LERW_natural, LV_lerw_chordal_note} we found a sequence $\ee_N \to 0$ as $N \to \infty$ (which depends on $D,a,b,K$) and for each $N$ sufficiently large we constructed $(\check \eta, \check \gamma)$ on the same probability space $(\Omega, \Prob)$ so that the following holds.
\begin{itemize}
\item{The marginal distribution on $\check \eta$ is LERW in $\check D$ from $\check a$ to $\check b$, parametrized as in \eqref{lerw-param}. Write $\F_{n}^{\LERW}$ for the filtration of generated by the LERW up to step $n$.
\begin{itemize}
\item{
To be more precise, we start from
a probability space on which is defined an infinite
sequence of uniform $[0,1]$ random variables $\{U_j\}$. 
Then $\F_{n}^{\LERW}$ can be taken to be the filtration generated by
by $U_1,\ldots,U_n$. The probabilities for
the at most three possible moves on the $(n+1)$:st
step of the LERW are determined by the walk up to time
$n$ (that is, is $\F_n^\LERW$-measurable) and the random
variable $U_{n+1}$ is used to make the choice in the usual manner. We also let $\tilde \F^{\LERW}_{n}$ denote the future
sigma algebra generated by $U_{n+1},U_{n+2},\ldots$. }
\end{itemize}}	
\item{The marginal distribution on $\check \gamma$ is chordal SLE$_2$ in $\check{D}$ from $\check a$ to $\check b$, parametrized by $5/4$-dimensional Minkowski content. The SLE is generated by a standard Brownian motion $W_t$ (and then reparametrized) and we write $\F_t^\SLE$ for the natural filtration of this Brownian motion, equivalently the filtration of $\check \gamma$ run up to capacity $t$. We let $\tilde \F_t^\SLE = \sigma\{W_{t+s} - W_t: \,s \ge 0 \}$ be the associated future sigma algebra. }
\item{
There exists a universal $v > 0$ such that if
\[
h=\lfloor N^{-v} \rfloor, 
\]
then here is a sequence of $\F^\LERW$-stopping times $\{m_{k}\}$ and a sequence of $\F^\SLE$-stopping times $\{\tau_k\}$ where $k=0, 1,\ldots, \lceil 2K/ h \rceil $ and an event $V$ for which the following holds.
\begin{itemize}
\item{If $T_{k} = c_{*}^{-1}m_{k}/N^{5/4}$ and $t_k = \Theta(\tau_k)$ is the $5/4$-dimensional Minkowski content of $\check \gamma$ run up to capacity $\tau_k$, then on the event $V$ we have
\[
\max_{k \le \lceil 2K/\delta \rceil }|T_{k} - t_{k}| \le \ee_{N}.
\]
Moreover, on $V$,
\[
\max_{k \le \lceil 2K/ h \rceil }|\hcap \check \gamma_{t_k} - kh| \le  \ee_N, \qquad \max_{k \le \lceil 2K/ h \rceil }|\check \eta(T_{k}) - \check \gamma(t_{k})| \le \ee_N.
\]
(The curve $\check \eta(t)$ is defined by linear interpolation for other times than $T_k$.)
}
\item{We have $\Prob(V) \ge 1-\ee_N$.}
\end{itemize}}
\item{The coupling has a Markovian property at the stopping times that can be phrased as follows. There is a filtration $\mathcal{G}_{k}$ such that  
\[  \F_{\tau_{k}}^\SLE \wedge  \F_{m_{k}}^\LERW \subset 
 \G_k , \]
 \[       \tilde \F_{\tau_{k}}^\SLE \wedge
      \tilde \F_{m_{k}}^\LERW  \perp  \G_k.\]
 We allow $\G_{k}$ to contain extra randomness but it must be independent of the future.
 }
 
 \end{itemize}
We now consider this coupling. For $r > 0$, define the stopping times
\[\sigma_r^\SLE = \inf\{t \ge 0 : |\check \gamma (t)| \le r\}, \quad \sigma_r^\LERW = \min\{j \ge 0: |\check \eta_j| \le r \}.
\]
Given $\delta, \delta' > 0$, define the stopping times
\[   \xi = \inf\{t \ge 0: S_t \leq \delta \mbox{ or }
      |\check \gamma(t)| \leq r \},
       \] 
 \[   \xi' = \inf\{t \geq \xi: \hcap[\check \gamma_{t} ]
  =  
   \hcap[\check \gamma_\xi] + \delta'\}.\]
Using Lemma~\ref{lemma1.nov9} and Lemma~\ref{lemma2.nov9} we may choose $\delta,\delta',c'$ bounded away from $0$ (depending on $r$) such
that if $E$ is the event that
\[           \min_{0 \leq t \leq \xi'} S_t \geq \delta,\;\;\;\;   
        \dist(0, \check\gamma_{\xi'}) \geq r/2,\]
        \[  h_{\check D \setminus \check \gamma_{\xi'}}
       (0,\check b) \geq c' \,  h_{\check  D}
       (0,\check b),\]
then,
       \[\mu^{\rad} ( E ) \ge 1- r.\]       
       Moreover, if $N$ is sufficiently large, then $h < \delta'$.
For the remainder, we only consider $N$ this large.
 
Let $J$ be the minimum of $\lceil K/h
  \rceil$ and the first $k$ with $t_k > \xi$. (Recall that $t_k = \Theta(\tau_k)$.)
The  care in the coupling was taken in order
to guarantee
that $J$ is a stopping time for the coupled
processes.  Also,   
$t_J < \xi'$ since $h < \delta'$.   
  
 We define the measure $\Q^\SLE$ on $\G_{J}$ by the relation
 \[    d\Q^\SLE =  M_{\tau_J}^\SLE \, d\Prob. 
 \]
 and let $Q_1^\SLE,Q_2^\SLE$ denote the induced marginal
 measure on $\check \gamma(t), 0 \leq t \leq t_J$ and 
 $\check \eta(t), 0 \leq t \leq T_J$, respectively.
 By construction $Q_1^\SLE$ is the 
 distribution of radial SLE$_2$, parametrized
 naturally,
 stopped at half-plane capacity $\tau_J \approx Jh$.

If $K=K_{r}$ is chosen sufficiently large, except on an event of $Q_1^\SLE$-probability $O(r)$,
 \[    \sigma_{r} \leq t_{J} \leq \sigma_{r/2}, \]
 \[       \min_{0 \leq t\leq t_J}	
  S_t  \geq \delta  .\]
  (This uses Lemma~\ref{lemma1.nov9} and Lemma~\ref{lemma2.nov9}.) Call $E^{\SLE}$ this event that the last two estimates hold. Note that on $E^{\SLE}$, which is $\G_{J}$ measurable, $M^{\SLE}_{\tau_{J}}$ is bounded by a constant depending only on $r,\delta$, see Lemma~\ref{lemma2.nov9}.
 We similarly define $\Q^\LERW $ on $\G_{J}$ by the relation
 \[   d\Q^\LERW  =  M_{m_J}^\LERW \, d\Prob.\]
 The paths $\check \eta$ under the marginal
 $ Q^\LERW _2$ have the distribution
 of scaled radial LERW parametrized naturally, stopped at $T_J$. Let $E^{\LERW}$ be the event that 
 \[
 \min_{k \le m_{J}} S^{\LERW}_{k} \ge 2\delta, \quad \sigma^{\LERW}_{r} \le m_{J} \le \sigma^{\LERW}_{r/2}.
 \]
 Then using Lemma~\ref{aug7.lemma3} we see that \begin{equation}\label{LERW-good}\Q^{\LERW}(E^{\LERW}) \ge 1-O(r).\end{equation} 
The $\Q^{\SLE}$-marginal on $\check \eta$ is not the LERW distribtion, but almost. We will make this precise by estimating the total variation distance between $\Q^\SLE$ and $\Q^\LERW$. 
 \begin{lemma}
 There are constants $c_{r,\delta}, c < \infty$ such that for $N$ sufficiently large,
 \[
\|\Q^\LERW - \Q^\SLE \| \le c_{r, \delta }\ee_N + c \, r.
\]
 \end{lemma}
 \begin{proof}
 Let $U \in \G_J$ be an arbitrary event. Then
\begin{align*}
|{\bf Q}^\LERW(U) - {\bf Q}^\SLE(U)|  \le & \, |{\bf Q}^\LERW(U\cap V \cap E^\SLE) - {\bf Q}^\SLE(U \cap V \cap E^\SLE)| \\
&  + {\bf Q}^\SLE(V^c \cap E^\SLE) +  {\bf Q}^\SLE((E^\SLE)^c) \\
& + {\bf Q}^\LERW(V^c \cap E^\LERW) + {\bf Q}^\LERW( (E^\LERW)^c)  
 \\ & + {\bf Q}^\LERW(V \cap (E^\SLE)^c)
\end{align*}
We estimate the terms on the right-hand side in order. First, on the event $V \cap E^\SLE$, the paths and Loewner chains are close and if $N$ is sufficuently large (depending on $r,\delta$) we can use Lemma~\ref{lerwrnlemma} to see that on this event,
\[
|M^\LERW_{n_J} - M^\SLE_{\tau_J}| \le c_{r,\delta} \ee_N, 
\]
so 
\[
|{\bf Q}^\LERW(U\cap V \cap E^\SLE) - {\bf Q}^\SLE(U \cap V \cap E^\SLE)| \le c_{r,\delta} \ee_N.
\]
Next, on the event $E^\SLE$ we have that $M^\SLE_{\tau_J}$ is bounded by a constant depending only on $r,\delta$. So using the fact that $\Prob(V^c) \le \ee_N$, we get
\[
{\bf Q}^\SLE(V^c \cap E^\SLE) \le c_{r,\delta} \ee_N
\]
We know that 
\[
{\bf Q}^\SLE((E^\SLE)^c) \le  c r.
\]
Similarly, using that $M^\LERW_{n_J}$ is bounded by $c_{r,\delta}$ on $E^\LERW$,
\begin{align*}
 {\bf Q}^\LERW(V^c \cap (E^\LERW)) & \le c_{r,\delta} \ee_N.
\end{align*}
Using  \eqref{LERW-good} we have
\[
\Q^\LERW((E^\LERW)^c) \le c r.
\]
And finally,  $E^\LERW \cap V \subset E^\SLE \cap V$, if $N$ is large enough so 
\[
{\bf Q}^\LERW(V \cap (E^\SLE)^c ) \le \Q^\LERW((E^\LERW)^c) \le  cr.
\]
We conclude that
\[
\|\Q^{\operatorname{LERW}} - \Q^\SLE \| \le c_{r, \delta }\ee_N + c r.
\]
\end{proof}
Hence if we take $N$ sufficiently large we see that the variation distance is at most $c r$, as claimed. It follows that for such $N$, the marginal of $\check \eta$ under $\Q^\SLE$ is within variation distance $cr$ of $Q^\LERW_2$. We take $\tau = m_{J}$ and $E$ to be the event $E^{\SLE}$. This completes the proof of Proposition~\ref{prop:main.intro}.

\section{Proof of Lemma~\ref{lerwrnlemma}}\label{sect.lemma1proof}
We recall the statement.
 \begin{lemma*}Let $r, \delta > 0$ be given. For all $A, \eta$ with $\eta \in I_{s}^{*}$ such that
   \begin{equation}\label{condition1111}
s \ge r N, \quad  \min_{j \leq n}
   S_{A_{j},a_{j},b} \ge \delta, \quad A_{j}=A \setminus \eta[0, \ldots \eta_{j}],
\end{equation}
we have
\[\frac{H_{A}(0,b)}{H_{{A_{\eta}}}(0,b)} = \frac{\Im F(0)}{\Im F_{\eta}(0)}\left(1+O_{r, \delta}(N^{-1/20}) \right).\]
Here  $F_{\eta}: D_{A_{\eta}} \to \mathbb{H}$ is the conformal map with $F_{\eta}(a_{\eta}) = 0, F_{\eta}(b) = \infty$ and such that $F_{\eta} \circ F^{-1}(z) = z + o(1)$  as $z \to \infty$. The error estimate is uniform over $A,\eta$ satisfying \eqref{condition1111}.
\end{lemma*}
\begin{proof}
We start by considering the smaller domain $A_{\eta} = A \setminus \eta$. By a last exit decomposition we have $4 H_{A_{\eta}}(0,b) = G_{A_{\eta}}(0,b^{*})$. We will write $G_{\eta} = G_{A_{\eta}}$. Let $\psi: D_{A} \to \mathbb{D}, \, \psi(0) =0, \, \psi'(0)>0$ and $\psi_{\eta} : D_{A_{\eta}} \to \mathbb{D}$ with $\psi_{\eta}(0) = 0, \psi_{\eta}'(0)>0$. Write $\zeta = \psi(b)$. It is not hard to show using, e.g., the Loewner equation, that the condition on the sine in \eqref{condition1111} implies that there is a constant $c_{\delta} > 0$ such that uniformizing ``Loewner map'' satisfies $|(\psi_{\eta} \circ \psi^{-1})'(\zeta)| \ge c_{\delta}$.
Define
\[
A^{*} = \{z \in A: g_{\eta}(0,z)  \ge N^{-1/16} \}, 
\]
where $g_\eta(z,w)$ is the Green's function for Brownian motion in the simply connected domain $D_{A_{\eta}}$. We let \[Q=Q(b) = \{z \in A_{\eta} : |\psi_\eta(z) - \psi_\eta(b)| \le c_0N^{-1/16}\log N \} \cap (A_{\eta} \setminus A^*),\]
where we will specify $c_0 > 0$ in a moment. 
Consider simple random walk $S$ and its exit time of $Q$:
\[
\sigma = \min\{ j \ge 0: S_j \in Q^c \}.
\]
We write $Q_{\text{top}} = \partial Q \cap A^*$, $Q_{\text{sides}} = \partial Q \cap A_{\eta}$. By Lemma~3.11 of \cite{KL} it is possible to choose $c_0 < \infty$ depending only on $r, \delta$ so that
\begin{equation}\label{side-top}
\frac{H_Q(b^*,Q_\text{sides})}{H_Q(b^*,Q_\text{top})} \le N^{-2}
\end{equation} for $N$ sufficiently large and we assume $c_0$ is chosen in this way and that $N$ is large enough that \eqref{side-top} holds.

By the strong Markov property,
\begin{align*}
G_{\eta}(0,b^{*}) &= \sum_{z \in Q_{\text{top}}} G_{\eta}(z,0) \Prob^{b^{*}}\{S_{\sigma} = z\}  +  \sum_{z \in Q_{\text{sides}}} G_{\eta}(z,0) \Prob^{b^{*}}\{S_{\sigma} = z\}. 
\end{align*}
By a rough bound on the Green's function, there is a constant $c$ such that
\[
\sum_{z \in Q_{\text{sides}}} G_{\eta}(z,0) \Prob^{b^{*}}\{S_{\sigma} = z\} \le  c H_Q(b^*,Q_\text{sides}).
\]
When $z \in A^*$  we can apply Theorem~1.2 of \cite{KL} to see that 
\begin{equation}\label{KLgreen}
G_{\eta}(0,z) = (2/\pi)g_{\eta}(0,z)[1+O_{r, \delta}(N^{-1/4}) ].
\end{equation}
Hence, $\sum_{Q_{\text{top}}} G_{k}(0,z) \Prob^{b^*}\{S_\sigma =z\} \ge c_{r, \delta} N^{-1/16}H_Q(b^*, Q_{\text{top}})$.  Consequently, using \eqref{side-top} 
\[G_{\eta}(0,b^{*})  =  \sum_{z \in Q_{\text{top}}} G_{\eta}(z,0) \Prob^{b^{*}}\{S_{\sigma} = z\} \left[ 1 + O_{r,\delta}(N^{-1})\right].\]
Let $\hat{g} = F_{\eta} \circ F^{-1}$ and $w = F(0)$. Using conformal invariance we have the formula
\[
g_{\eta}(0,z) = \log \left| \frac{\hat g \circ F(z)-\hat g(w)}{\hat{g} \circ F(z)-\overline{\hat g(w)}} \right| = 2\frac{\Im F_{\eta}(0)}{|F(z)|^{2}/\Im F(z)}[1+O((\Im F(z))^{-1})],
\]
where the asymptotic expansion holds for $\Im F(z)$ large and we used the normalization of $\hat{g}$ at $\infty$. For $z \in \partial Q$ we have $\Im F(z) \ge C_r N^{1/16}/(\log N)^{2}$ and so by \eqref{KLgreen}
\begin{align*}
G_{\eta}(0,b^{*}) & =  \sum_{z \in Q_{\text{top}}} G_{\eta}(z,0) \Prob^{b^{*}}\{S_{\sigma} = z\} \left[ 1 + O_{r, \delta}(N^{-{1/4}})\right] \\
& =(4/\pi)\sum_{z \in Q_{\text{top}}}  \frac{\Im F_{\eta}(0)}{|F(z)|^{2}/\Im F(z)}  \Prob^{b^{*}}\{S_{\sigma} = z\}  [1+O_{r, \delta}(N^{-1/20}) ].
\end{align*}
We now consider the larger domain $A \supset A_{\eta}$. We keep the sets $Q, Q_{\text{top}}, Q_{\text{sides}}$  which are all contained in $A$ or in $\partial A$ for $N$ sufficiently large. We carry out the same argument as for $A_{\eta}$. The estimates \eqref{KLgreen} and \eqref{side-top} hold with $A_{\eta}$ replaced by $A$; the former using monotonicity of the Green's function and the latter holds by the assumption \eqref{condition1111}. The conclusion is that
\begin{align*}
G_{A}(0,b^{*}) & =  \sum_{z \in Q_{\text{top}}} G_{A}(z,0) \Prob^{b^{*}}\{S_{\sigma} = z\} \left[ 1 + O_{r, \delta}(N^{-{1/4}})\right] \\
& =(4/\pi)\sum_{z \in Q_{\text{top}}}  \frac{\Im F(0)}{|F(z)|^{2}/\Im F(z)}  \Prob^{b^{*}}\{S_{\sigma} = z\}  [1+O_{r, \delta}(N^{-1/20}) ].
\end{align*}
We divide this expression with the one for $G_{\eta}(0,b^{*})$ to see that
\[
\frac{G_{A}(0,b^{*})}{G_{\eta}(0,b^{*})}=\frac{\Im F(0)}{\Im F^{\eta}(0)}[1+O_{r, \delta}(N^{-1/20})],
\]
which, using the last-exit decomposition, is what we wanted to prove.
\end{proof}

\end{document}